\documentclass[11pt]{article}

\usepackage{amsfonts,amsmath,amssymb}
\usepackage{amsmath,amsthm}
\usepackage{latexsym}
\usepackage{color}
\usepackage{tikz}

\newtheorem{thm}{Theorem}

\newtheorem{lem}[thm]{Lemma}
\newtheorem{conj}{Conjecture}
\newtheorem{cor}[thm]{Corollary}
\newtheorem{prop}[thm]{Proposition}
\newtheorem{ques}{Question}
\newtheorem{prob}{Problem}

\newcommand{\cp}{\,\square\,}

\newcommand{\LSs}[2]{^{#1}{#2}}

\begin{document}

\title{On independent domination in direct products}

\author{$^1$Kirsti Kuenzel and $^2$Douglas F. Rall \\
\\
$^1$Department of Mathematics \\
Trinity College \\
Hartford, Connecticut  USA \\
\small \tt Email: kwashmath@gmail.com\\
\\
$^2$Department of Mathematics \\
Furman University \\
Greenville, SC, USA\\
\small \tt Email: doug.rall@furman.edu}

\date{\today}

\maketitle

\begin{abstract}
In \cite{nr-1996} Nowakowski and Rall listed a series of conjectures involving several different graph products. In particular, they conjectured that $i(G\times H) \ge i(G)i(H)$ where $i(G)$ is the independent domination number of $G$ and $G\times H$ is the direct product of graphs $G$ and $H$. We show this conjecture is false, and, in fact, construct pairs of graphs for which $\min\{i(G), i(H)\} - i(G\times H)$ is arbitrarily large. We also give the exact value of $i(G\times K_n)$ when $G$ is either a path or a cycle.
\end{abstract}

\medskip
\noindent{\bf AMS Subj.\ Class:} 05C69; 05C76

\section{Introduction} \label{sec:introduction}

Independence in graph products has been studied by many authors but almost always in the context of the independence number, commonly denoted by $\alpha$.
We mention just samples  of papers concerning the independence number of a Cartesian product $\alpha(G \cp H)$ (see~\cite{f-2011, hk-1996, js-1994, k-2005, nr-1996})
and  of a direct product $\alpha(G \times H)$ (see~\cite{jk-1998, kr-2022, nr-1996}).  In addition, for both of these two products some investigation has also been done
on the so-called ultimate independence ratios,  $\lim_{m \to \infty}\frac{\alpha(\cp_{i=1}^m G)}{n(G)^m}$ and $\lim_{m \to \infty}\frac{\alpha(\times_{i=1}^m G)}{n(G)^m}$.
See for example~\cite{al-2007, bnr-1996, hyz-1994, t-2014}.

Nowakowski and Rall~\cite{nr-1996} studied the behavior of a number of domination, independence and coloring type invariants on nine associative graph products whose edge sets depend on the edge sets of both factors.  In particular, they proved some lower and upper bounds for the
cardinality of a smallest maximal independent set, the independent domination number, of these products.  For an excellent survey of independent domination
see the paper~\cite{gh-2013} by Goddard and Henning.
In this work we will focus on the independent domination number of the direct product of two graphs.  In particular, we are interested
in how the independent domination number of a direct product relates to the independent domination numbers of the two factors.
In the process we give a counterexample to the following conjecture of Nowakowski and Rall.
\begin{conj} {\rm \cite[Section 2.4]{nr-1996}}  \label{conj:lowerbound}
For all graphs $G$ and $H$, $i(G \times H) \ge i(G)i(H)$.
\end{conj}
In fact, we prove a stronger result; namely
\begin{thm} \label{thm:extremecounterexample}
For any positive integer $n$ such that $n>10$, there exists a pair of graphs $G$ and $H$ such that
$\min\{i(G),i(H)\}=n+2$ and $i(G \times H) \leq 12$.
\end{thm}

The organization of the paper is as follows.  In the next section we provide necessary definitions and several previous results.
In Section~\ref{sec:productwithcomplete} we restrict our attention to direct products in which one of the factors is a complete
graph, and introduce a method for calculating the independent domination number of $G \times K_n$ in terms of minimizing
a certain kind of labelling of $V(G)$.  Using this scheme we find the values of $i(P_m \times K_n)$ and $i(C_m \times K_n)$.
Lower bounds for $i(G \times H)$, in terms of other domination-type invariants of $G$ and $H$, are given in Section~\ref{sec:lowerbounds}.
The main result of the paper is in Section~\ref{sec:counterexamples} where we give an infinite collection of counterexamples to
Conjecture~\ref{conj:lowerbound} and prove Theorem~\ref{thm:extremecounterexample}.

\section{Definitions and preliminary results} \label{sec:defns}

We denote the order of a finite graph $G=(V(G),E(G))$ by $n(G)$.  For a positive integer $n$ we let $[n]=\{1,\ldots,n\}$; the vertex set of the complete graph $K_n$
will be $[n]$ throughout.  A subset $D\subseteq V(G)$ \emph{dominates} a subset $S \subseteq V(G)$ if $S \subseteq N[D]$.  If $D$ dominates~$V(G)$, then we will also say that $D$ dominates the graph $G$ and that $D$ is a \emph{dominating set} of $G$.  If $D$, in addition to being a dominating set
of $G$, has the property that every vertex in $D$ is adjacent to at least one other vertex of $D$, then $D$ is a \emph{total dominating set} of $G$. The
\emph{total domination number} of $G$ is the minimum cardinality among all total dominating sets of $G$; it is denoted $\gamma_t(G)$.  The \emph{$2$-packing number} of $G$, denoted $\rho(G)$, is the largest cardinality of a vertex subset $A$ such that the distance in $G$ between $a_1$ and $a_2$ is at least $3$ for every pair $a_1,a_2$
of distinct vertices in $A$.
A set $I \subseteq V(G)$ is an \emph{independent dominating} set if $I$ is simultaneously independent and dominating.  This is equivalent to $I$ being a maximal independent set with respect to set inclusion.  The \emph{independence number} of $G$ is the cardinality, $\alpha(G)$, of a largest independent set in $G$.  We denote by $i(G)$ the smallest cardinality of a maximal independent set in $G$; this invariant is called the \emph{independent domination number} of $G$.

The \emph{direct product}, $G\times H$, of graphs $G$ and $H$ is defined as follows:
\begin{itemize}
\item $V(G \times H)=V(G) \times V(H)$;
\item $E(G \times H)= \{ (g_1,h_1)(g_2,h_2) \, \colon\, g_1g_2 \in E(G) \,\,\text{and}\,\,h_1h_2 \in E(H) \}$
\end{itemize}
The direct product is both commutative and associative.   For a vertex $g$ of $G$, the \emph{$H$-layer over $g$} of $G\times H$ is the set $\{ \, (g,h) \mid h\in V(H) \,\}$, and it is denoted by $\LSs g H$.  Similarly, for $h \in V(H)$, the \emph{$G$-layer over $h$}, $G^h$, is the set $\{ \, (g,h) \mid g\in V(G) \,\}$.  Note that each $G$-layer and each $H$-layer is an independent set in $G\times H$.  The \emph{projection to $G$} is the map $p_G: V(G\times H) \to V(G)$ defined by $p_G(g,h)=g$.  Similarly, the \emph{projection to $H$} is the map $p_H: V(G\times H) \to V(H)$ defined by $p_H(g,h)=h$.  If $A \subseteq V(G\times H)$ and $g \in V(G)$, then we employ
$\LSs g A$ to denote $A \cap\, \LSs g H$.  Similarly, $A^h=A \cap G^h$ for a vertex $h$ of $H$.

The following result of Topp and Volkmann will be useful in establishing our main results.
\begin{lem} {\rm \cite[Proposition 11]{tv-1992}} \label{lem:inverseimage}
Let $H$ be a graph with no isolates.  If $I$ is a maximal independent set of any graph $G$, then $I \times V(H)$ is a maximal independent set of $G \times H$.
\end{lem}

\medskip
As an immediate consequence of Lemma~\ref{lem:inverseimage} we get a lower bound for $\alpha(G \times H)$, which is well-known, and an upper bound for $i(G \times H)$.
Both were established earlier by Nowakowski and Rall~\cite{nr-1996}.
\begin{cor} {\rm \cite[Table 3]{nr-1996}}            \label{cor:triviallower}
If both $G$ and $H$ have no isolated vertices, then
\begin{itemize}
\item $\alpha(G \times H) \ge \max \{\alpha(G) n(H), \alpha(H) n(G)\}$;
\item $i(G \times H) \le \min \{i(G) n(H), i(H) n(G)\}$.
\end{itemize}
\end{cor}

\section{Independent domination in $G \times K_n$} \label{sec:productwithcomplete}

In this section we focus on direct products in which one of the factors is a complete graph, and we will use notation introduced in our paper~\cite{kr-2022}.

Let $I$ be a maximal independent set of $G \times H$.  Suppose $g$ is a vertex of $G$ such that $\LSs{g}{I} \not=\emptyset$ but  $\LSs{g}{I} \not= {\LSs{g}{H}}$.  Let $(g,h)\in {\LSs{g}{H}- \,\LSs{g}{I}}$.  Since $I$ is a dominating set of $G \times H$, it follows that there exists $g' \in N_G(g)$ and $h'\in N_H(h)$ such that $(g',h') \in I$.  Note that  such a vertex $h'$ does not belong to $N_H(p_H(\LSs{g}{I}))$.  For if $h'x \in E(H)$ for some $(g,x) \in I$, then $(g',h')$ and $(g,x)$ are adjacent vertices of $I$, which
is a contradiction.  However, it is possible that $h' \in p_H(\LSs{g}{I})$.

Consider now the special case $G \times K_n$ for $n \ge 2$.  The following lemma is from~\cite{kr-2022}.  For the sake of completeness we give its short proof.

\begin{lem} {\rm \cite[Lemma 9]{kr-2022}}\label{lem:sizeoflayers}
Let $n \ge 2$ and let $G$ be any graph.  If $I$ is any maximal independent set of $G \times K_n$, then $\left| I \cap {\LSs{g}{K_n}} \right | \in \{0,1,n\}$, for any $g \in V(G)$.
\end{lem}
\begin{proof} If $n=2$, then the conclusion is obvious.  Assume $n\ge 3$ and suppose for the sake of contradiction that $\left| I \cap {\LSs{g}{K_n}} \right |=m$ for some $2 \le m <n$.  Assume without loss of generality that $\{(g,1),(g,2)\} \subset I$.  Let $i \in [n]$ such that $(g,i) \not\in I$. As above, there exists $g' \in N_G(g)$ and $j \in N_{K_n}(i)$ such that $(g',j)\in I$.  Since $n \ge 3$, we infer that $j \not=1$ or $j \not =2$.  This implies that $(g',j) \in N(\{(g,1),(g,2)\})$, which contradicts the independence of $I$.  Therefore,
$\left| I \cap {\LSs{g}{K_n}} \right | \in \{0,1,n\}$.
\end{proof}

The following result gives tight upper and lower bounds for $i(G \times K_2)$.
\begin{thm} \label{thm:lowerandupper2}
If $G$ is any graph with no isolated vertices, then
\[\gamma_t(G) \le i(G \times K_2) \le \min\{2i(G),n(G)\}\,.\]
\end{thm}
\begin{proof}  The upper bound follows from Corollary~\ref{cor:triviallower}.  Let $M$ be an independent dominating set of $G \times K_2$ such that $i(G \times K_2)=|M|$.  If $M^1$ or $M^2$ is empty, say $M^2=\emptyset$, then $(g,1) \in M$ for every $g\in V(G)$.  Hence, $|M|=n(G)\ge \gamma_t(G)$.  Thus, assume that $M^1\not= \emptyset$ and $M^2\not=\emptyset$.  For each $(g,2) \in M$, choose $g' \in N_G(g)$.  Let $\widehat{M}= M^1 \cup \{(g',1)\,\colon\,(g,2) \in M\}$.  It is clear that $\widehat{M}$ dominates $G^2$.  For if $(g,2) \in M$, then $(g',1) \in \widehat{M}$ and $(g',1)$ is a neighbor of $(g,2)$.  If $(g,2) \not\in M$, then $M^1$ contains a neighbor of $(g,2)$ since $M$ is a dominating set of $G \times K_2$.  Consequently, $p_G(\widehat{M})$ is a total dominating set of $G$, and we get
\[\gamma_t(G) \le |p_G(\widehat{M})| \leq |\widehat{M}| \le |M|=i(G \times K_2)\,.\]
\end{proof}

Any graph $G$ that has a vertex of degree $n(G)-1$ shows that the lower bound in Theorem~\ref{thm:lowerandupper2} is tight.  For the upper bound let $G=K_{n,n}$.  Since $G \times K_2 =2G$, we see that the upper bound is also tight.

Let $I$ be any maximal independent set of $G \times K_n$ and let $n\ge 2$ be a positive integer.  As in~\cite{kr-2022} we use Lemma~\ref{lem:sizeoflayers} to define a weak partition of $V(G)$.  We will say this weak partition is \emph{generated by} or \emph{corresponds to} $I$.
(A weak partition of a set $X$ is a collection of pairwise disjoint subsets of $X$, in which some may be empty, whose union is $X$.)  In particular, $V_0,V_1,\ldots,V_n,V_{[n]}$ defined by
\begin{enumerate}
\item[{\bf(a)}] $V_0=\{ g \in V(G)\,\colon\, I \cap {\LSs{g}{K_n}} = \emptyset\}$;
\item[{\bf(b)}] For each $k \in [n]$, $V_k=\{ g \in V(G)\,\colon\, I \cap {\LSs{g}{K_n}} =\{(g,k)\} \}$;
\item[{\bf(c)}] $V_{[n]}=\{ g \in V(G)\,\colon\, I \cap {\LSs{g}{K_n}}={\LSs{g}{K_n}} \}$.
\end{enumerate}
is a weak partition.  Furthermore, the following four conditions hold.
\begin{enumerate}
\item For $k \in [n]$, if $u \in V_k$ and $v \in V(G)-(V_0 \cup V_k)$, then $uv \not\in E(G)$.
\item For $k \in [n]$, if $V_k$ is not empty, then no vertex of $V_k$ is isolated in $G[V_k]$.
\item The set $V_{[n]}$ is independent in $G$.
\item For each $g \in V_0$, either $N_G(g) \cap V_{[n]} \not=\emptyset$ or $g$ has a neighbor in at least two of the sets $V_1,\ldots,V_n$.
\end{enumerate}

Conversely, for a given weak partition of $V(G)$ that satisfies these four conditions, it is clear how to construct a maximal independent set $D$ of $G \times K_n$.
We then say this weak partition \emph{constructs} $D$.  The independent domination number of $G \times K_n$ can be computed in the following way.
\begin{equation} \label{eqn:indepdomlabel}
i(G \times K_n)= \min\{n\cdot |V_{[n]}|+ \sum_{k=1}^n|V_k|\}\,,
\end{equation}
where the minimum is computed over all weak partitions $V_0,V_1,\ldots,V_n,V_{[n]}$ that satisfy conditions $1{-}4$ above.

We used a computer program to compute the independent domination numbers of the direct product of small paths and small cycles with complete graphs.  In Table~\ref{Tab:Special} we accumulate some of these values.  Because of the smallest independent dominating sets produced by our software, it is clear that these same values hold if $K_3$ is replaced by $K_n$ in the direct product for any $n \ge 4$.

\begin{table}[ht!]
\begin{center}
\begin{tabular} { |c | c c c c c c c c c c | }
\hline
$m$  &3&4&5&6&7&8&9&10&11&12 \\ \hline
$i(P_m \times K_3)$ &3&4&4&5&6&6&7&8&8&9 \\ \hline
$i(C_m \times K_3)$ &3&4&5&4&5&6&6&7&8&8 \\ \hline
\end{tabular}
\caption{Some independent domination numbers}\label{Tab:Special}
\end{center}
\end{table}

Instead of specifically listing the sets $V_0,V_1,\ldots,V_n,V_{[n]}$ we can instead represent this weak partition by a labelling of the vertices of $G$.
A vertex $x$ of $G$ is labelled with the symbol $\ast \in \{0,1,\ldots,n,[n]\}$ if and only if $x \in V_{\ast}$.  (Note that we are allowing labels
to  be used more than once.) For any weak partition
$V_0,V_1,\ldots,V_n,V_{[n]}$ (equivalently any labelling) of $V(G)$ that satisfies the four conditions, we say it is  \emph{legal} and has
\emph{weight} $n\cdot |V_{[n]}|+ \sum_{k=1}^n|V_k|$.  For the purposes of this paper, we
call any weak partition (equivalently, any labelling) of $V(G)$ \emph{optimum} if it attains the minimum weight in~\eqref{eqn:indepdomlabel} above.
We illustrate this labelling in Figure~\ref{fig:labeling}
on a cycle of order $16$ that defines a maximal independent set of the direct product $C_{16} \times K_3$.  Note that $16=6r+p$ for $r=2$ and $p=4$.
This labelling is part of the pattern denoted by $(1,1,0,2,2,0)^2(3,3,3,0)$, which means label the vertices of $C_{16}$ in consecutive order by
repeating the sequence of six  labels $1,1,0,2,2,0$ two ($r=2$) times followed by the sequence of four ($p=4$) labels $3,3,3,0$ one time.

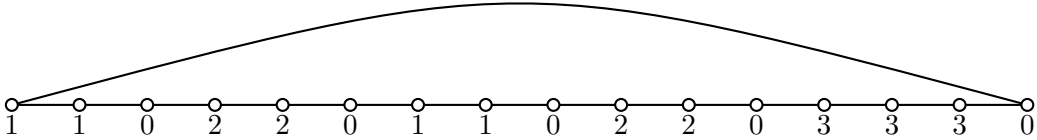
\begin{figure}[ht!]
\begin{center}
\begin{tikzpicture}[scale=0.45,style=thick]
\def\vr{5pt} % \vr = vertex radius;  Set \vr = 2/scale for uniform sizing of vertices
%%%
%%% Vertices:
\path (-16,0) coordinate (0); \path (-14,0) coordinate (1);
\path (-12,0) coordinate (2);  \path (-10,0) coordinate (3);
\path (-8,0) coordinate (4);  \path (-6,0) coordinate (5);
\path (-4,0) coordinate (6);  \path (-2,0) coordinate (7);
\path (0,0) coordinate (8);  \path (2,0) coordinate (9);
\path (4,0) coordinate (10);  \path (6,0) coordinate (11);
\path (8,0) coordinate (12);  \path (10,0) coordinate (13);
\path (12,0) coordinate (14);  \path (14,0) coordinate (15);
\draw (0) .. controls (-1,4) .. (15);
%%% Vertices: Horizontal
%%% Draw edges
\draw (0) -- (1); \draw (1) -- (2);  \draw (2) -- (3);  \draw (3) -- (4);  \draw (4) -- (5);  \draw (5) -- (6);
 \draw (6) -- (7);  \draw (7) -- (8);  \draw (8) -- (9);  \draw (9) -- (10);  \draw (10) -- (11);  \draw (11) -- (12);
 \draw (12) -- (13);  \draw (13) -- (14);  \draw (14) -- (15);

\foreach \i in {0,...,15}
{  \draw (\i)  [fill=white] circle (\vr); }
\foreach \i in {2,5,8,11,15}
\draw[anchor = north] (\i) node {$0$};
\foreach \i in {0,1,6,7}
\draw[anchor = north] (\i) node {$1$};
\foreach \i in {3,4,9,10}
\draw[anchor = north] (\i) node {$2$};
\foreach \i in {12,13,14}
\draw[anchor = north] (\i) node {$3$};
%\draw[anchor = north] (0) node {$\mathbf x_1$}; \draw[anchor = north] (1) node {$\mathbf x_2$};
%\draw[anchor = east] (13) node {$\mathbf a$}; \draw[anchor = east] (14) node {$\mathbf b$};
%\draw[anchor = east] (15) node {$\mathbf c$};
\end{tikzpicture}
\end{center}
\caption{Labeling of $C_{16}$}
\label{fig:labeling}
\end{figure}

\begin{prop} \label{prop:pathscycles}
Let $m$ and $n$ be positive integers with $m \ge 3$ and $n \ge 3$.  Then,
\begin{itemize}
\item[{\rm(a)}] $i(P_m \times K_2)=2\lceil \frac{m}{3} \rceil$.
\item [{\rm(b)}] $i(C_m \times K_2)= \lceil \frac{2m}{3} \rceil$ if $m$ is odd, and $i(C_m \times K_2)= 2\lceil \frac{m}{3} \rceil$ if $m$ is odd.
\item [{\rm(c)}] $i(C_m \times K_n)=m$ for $3 \le m \le 5$, and $i(C_m \times K_n)=\left\lceil \frac{2m}{3} \right \rceil$, for every $m \ge 6$.
\item [{\rm(d)}] $i(P_m \times K_n)=\left\lceil \frac{2m+2}{3} \right \rceil$.
\end{itemize}
\end{prop}
\begin{proof} Note that $C_m \times K_2=C_{2m}$ if $m$ is odd, and $C_m \times K_2=2C_{m}$ if $m$ is even.  Also, for any $m$,
$P_m \times K_2=2P_m$.   Statements (a) and (b) now follow from $i(C_k)=i(P_k)=\lceil k/3 \rceil$.

Now consider $C_m \times K_n$ for some $n \ge 3$.  It is easy to check that the vertex labellings $(1,1,1)$, $(1,1,1,1)$ and
$(1,1,1,1,1)$  of $C_3$, $C_4$ and $C_5$ respectively are optimum for the direct products $C_m \times K_n$ for $3 \le m \le 5$.
Now, let $m \ge 6$.  Table~\ref{Tab:Largerm} presents labelling patterns of $V(C_m)$ and $V(P_m)$, based on the congruence of $m$ modulo $6$,
that establish the upper bound of $\lceil 2m/3 \rceil$ for $i(C_m \times K_n)$ and of $\left\lceil \frac{2m+2}{3} \right \rceil$ for $i(P_m \times K_n)$.

We may assume that in any optimum labelling of $V(C_m)$ no vertex receives the label $[n]$.  For, suppose some vertex $x$ is labelled $[n]$. By conditions $1$ and $3$
both of the neighbors of $x$ are labelled $0$.  For the vertices of $C_m$ within distance $2$ of $x$, the labelling sequence
$(i,0,[n],0,i)$ can be replaced with $(i,i,0,i,i)$  and the sequence  $(i,0,[n],0,j)$, for $i \neq j$, can be replaced with $(i,i,0,j,j)$.
This in turn implies that at most one vertex of any  three consecutive
vertices of $C_m$ can be labelled $0$.  That is, at most $\lfloor m/3 \rfloor$ vertices can be labelled $0$.  Since
$m=\lfloor m/3\rfloor +\lceil 2m/3 \rceil$, we get $i(C_m \times K_n) \ge \lceil 2m/3 \rceil$.  This establishes statement (c).

We use essentially the same reasoning to prove the lower bound for $i(P_m \times K_n)$ with one small difference, this being that the path
has two vertices of degree $1$.  This forces several additional cases.  Suppose the path $P_m$ is $v_1,v_2,\ldots,v_{m-1},v_m$.
We claim that without loss of generality we may assume that any
optimum labelling of $V(P_m)$ does not use the label $[n]$.  We modify any labelling that has such a vertex $x$ with label $[n]$ in such a way
that the weight is not increased.  Note that every neighbor of $x$ is labelled $0$.  Suppose first that $x$ and the vertices within distance $2$
of $x$ all have degree $2$. As above,  the labelling sequence $(i,0,[n],0,i)$ can be replaced with $(i,i,0,i,i)$  and the sequence
$(i,0,[n],0,j)$, for $i \neq j$, can be replaced with $(i,i,0,j,j)$.  The labelling sequence $([n],0,[n],0,i)$ can be replaced with $([n],0,i,i,i)$;
similarly $(i,0,[n],0,[n])$ can be replaced with $(i,i,i,0,[n])$.  Now suppose that $x=v_2$.
The sequence $(0,[n],0,i,i)$ can be replaced with $(i,i,i,i,i)$, the sequence $(0,[n],0,[n],0)$ can be replaced with
$(i,i,0,[n],0)$, and $(0,[n],0,0,[n])$ can be replaced with $(i,i,i,0,[n])$.  The case $x=v_{m-1}$ is handled in a similar fashion.
Finally, if $0$ and $[n]$ are the only labels used, then the only case to consider is when one or both of the end vertices is labelled $[n]$. Suppose
$v_1$ is labelled $[n]$.  The labelling sequence $([n],0,[n],0,[n])$ can be replaced with $(i,i,i,0,[n])$,  the sequence $([n],0,[n],0,0)$
can be replaced with $(i,i,i,i,0)$, and $([n],0,0,[n],0)$ can be replaced with $(i,i,0,[n],0)$.  The case where $v_m$ is labelled $[n]$ is handled
symmetrically.  Thus, we may assume that no optimum labelling of $V(P_m)$ uses the label $[n]$.

Therefore, any optimum labelling of $V(P_m)$ has weight $\sum_{k=1}^n|V_k|$ since $V_{[n]}=\emptyset$.  By considering the three cases of $m$ modulo
$3$ and using conditions $1,2$ and $4$, it is now easy to see that at most $\left \lfloor \frac{m-2}{3} \right \rfloor$ vertices of $P_m$ can be labelled $0$.
Since $m=  \left \lfloor \frac{m-2}{3} \right \rfloor + \left \lceil \frac{2m+2}{3} \right \rceil$, it follows that any legal
labelling of $V(P_m)$ has weight at least $\left \lceil \frac{2m+2}{3} \right \rceil$.  This lower bound coincides with the values
given in Table~\ref{Tab:Largerm}, which finishes the proof.
\end{proof}
\begin{table}
\begin{center}
\begin{tabular} { |c | c | c |}
\hline
$m=$ & labelling pattern of $C_m$ & labelling pattern of $P_m$  \\ \hline
$6r$ & $(1,1,0,2,2,0)^r$ & $(1,1,0,2,2,0)^{r-1}(1,1,0,2,2,2)$\\ \hline
$6r+1$ & $(1,1,0,2,2,0)^{r-1}(1,1,0,2,2,2,0)$ & $(1,1,0,2,2,0)^{r-1}(1,1,1,0,2,2,2)$\\ \hline
$6r+2$ & $(1,1,0,2,2,0)^{r-1}(1,1,0,2,2,2,2,0)$ & $(1,1,0,2,2,0)^r(1,1)$\\ \hline
$6r+3$ & $(1,1,0,2,2,0)^r (3,3,0)$ & $(1,1,0,2,2,0)^r(1,1,1)$\\ \hline
$6r+4$ & $(1,1,0,2,2,0)^r (3,3,3,0)$ & $(1,1,0,2,2,0)^r(1,1,1,1)$\\ \hline
$6r+5$ & $(1,1,0,2,2,0)^r (3,3,3,3,0)$ & $(1,1,0,2,2,0)^r(1,1,0,2,2)$\\ \hline
\end{tabular}
\caption{Construction of minimum independent dominating sets}\label{Tab:Largerm}
\end{center}
\end{table}

\section{Lower bounds} \label{sec:lowerbounds}

In the section we prove some lower bounds for $i(G \times H)$ in terms of other domination-type graphical invariants.
Using an argument similar to that in the proof of Theorem~\ref{thm:lowerandupper2}, we can establish a lower bound for the independent
domination number of two graphs, neither of which has an isolated vertex.

\begin{prop} \label{prop:generallower}
If $G$ and $H$ are any two graphs that both have minimum degree at least 1, then
\[i(G\times H) \ge \max\{\rho(G)\gamma_t(H),\rho(H)\gamma_t(G)\}\,.\]
\end{prop}
\begin{proof}   Let $I$ be an independent dominating set of $G \times H$ of smallest cardinality. For a vertex $g$ of $G$, let $J_g=I \cap (N_G(g) \times V(H))$.  Since $I$ is a dominating set, each vertex of $\{g\} \times V(H)$ is either in $I$ or is adjacent to a vertex in $J_g$.  Furthermore, since $I$ is independent, exactly one of these holds for each vertex
in $\{g\} \times V(H)$.  For each $h\in V(H)$ such that $(g,h) \in I$, fix a single neighbor $(g',h')$ of $(g,h)$.  Finally, let
$\widehat{J}_g = J_g \cup \{(g',h')\,\colon\,(g,h) \in I\}$.  Note that $|\{(g',h')\,\colon\,(g,h) \in I\}|\le |\LSs{g}{I}|$ and that
the projection $p_H(\widehat{J}_g)$ is a total dominating set of $H$. Let $A$ be a maximum 2-packing of $G$.  It now follows that
\[|I|=\sum_{x \in V(G)}|\LSs{x}{I}| \ge \sum_{g\in A}|I \cap(N_G[g]\times V(H))|\ge \sum_{g\in A}|\widehat{J}_g|\ge \sum_{g\in A}|p_H(\widehat{J}_g)|\ge \rho(G)\gamma_t(H)\,.\]
By reversing the roles of $G$ and $H$ in the above argument we get the desired conclusion.
\end{proof}

In Section~\ref{sec:counterexamples} we demonstrate the existence of pairs of graphs $G$ and $H$ such that
$i(G \times H) < \min\{i(G),i(H)\}$.  The following corollary to Proposition~\ref{prop:generallower} shows that when
both factors are claw-free this is not possible.

\begin{cor}
If $G$ and $H$ are both claw-free with no isolated vertices, then
\[i(G \times H)\geq \max\{i(G),i(H)\}\,.\]
\end{cor}
\begin{proof}
From Proposition~\ref{prop:generallower} it follows directly that $i(G \times H)\geq \max\{\gamma(G),\gamma(H)\}$.
The result now follows  since $G$ and $H$ are claw-free, which implies that  $\gamma(G)=i(G)$ and $\gamma(H)=i(H)$.
\end{proof}

\begin{prop} For any connected graphs $G$ and $H$,
\[i(G\times H) \ge \max\left\{\frac{{n(H)}{\gamma(G)}}{\Delta(H)+1}, \frac{{n(G)}{\gamma(H)}}{\Delta(G)+1} \right\}\,.\]
\end{prop}

\begin{proof} Let $D$ be an independent dominating set of $G\times H$. For each $v \in V(H)$, let $X_v = p_G(D \cap(V(G) \times N_H[v]))$.
Note that $X_v$ is a dominating set of $G$. Moreover, $\sum_{v \in V(H)} |X_v|$ counts each vertex of $D$ at most $\Delta(H)+1$ times. Thus,
\[|D| \ge \frac{\sum_{v\in V(H)} |X_v|}{\Delta(H) + 1} \ge \frac{n(H)\gamma(G)}{\Delta(H)+1}.\]
The result now follows by interchanging the roles of $G$ and $H$.
\end{proof}

If both factors of a direct product are connected and bipartite, then we get a lower bound just in terms of the domination numbers of the factors.
\begin{prop} If $G$ and $H$ are two connected bipartite graphs, then
\[i(G\times H) \ge \max\{2\gamma(G), 2\gamma(H)\}\,.\]
\end{prop}

\begin{proof} Let $D$ be an independent dominating set of $G\times H$. Let $A_G,B_G$ be the bipartition of $V(G)$ and $A_H,B_H$ be the
bipartition of $V(H)$.  Note that $G\times H$ is disconnected since no vertex in $(A_G\times B_H) \cup (B_G \times A_H)$ is adjacent to a vertex in
$(A_G \times A_H) \cup (B_G \times B_H)$.  Thus, it suffices to show that
\[|D \cap ( (A_G\times B_H)\cup (B_G\times A_H))| \ge \gamma(G).\]
Let $A_2=p_G(D \cap (A_G \times B_H))$ and let $A_1=A_G-A_2$.  Similarly, we let $B_2=p_G(D \cap (B_G \times A_H))$ and let $B_1=B_G-B_2$.
We claim that $A_2 \cup B_2$ is a dominating set of $G$. To see this, fix $x \in A_1$.  Since $A_G\times B_H$ is an independent set,
every vertex in $\{x\} \times B_H$ is dominated by some vertex in $B_2 \times A_H$.  This implies that $B_2$ dominates $A_1$, and
a similar argument shows $A_2$ dominates $B_1$.  Thus, $A_2 \cup B_2$ is a dominating set of $G$, and we infer that
$|D \cap ( (A_G\times B_H)\cup (B_G\times A_H))|\geq \gamma(G)$.  Similarly, $|D \cap ( (A_G\times A_H)\cup (B_G\times B_H))| \geq \gamma(G)$,
and it follows that $|D| \geq 2\gamma(G)$.
Interchanging the roles of $G$ and $H$ in the above argument shows that $|D| \geq 2\gamma(H)$, which finishes the proof.
\end{proof}

\medskip

\section{Counterexamples to Conjecture~\ref{conj:lowerbound}} \label{sec:counterexamples}

We now present counterexamples to Conjecture~\ref{conj:lowerbound}.  Let $m$ and $r$ be positive integers larger than $2$.  Let $A,B,C$ and $D$ be pairwise disjoint
independent sets of cardinality $m$.  The graph $X_m$ has vertex set and edge set defined as follows.
\begin{itemize}
\item $V(X_m)=A\cup B\cup C \cup D \cup \{x_1,x_2,x_3,x_4\}$.
\item $E(X_m)=\{x_1w\,:\,w \in A\cup C\}\cup \{x_2w\,:\,w \in B\cup D\} \cup \{x_3w\,:\,w \in A\cup B\} \cup \{x_4w\,:\,w \in C\cup D\} \cup \{x_1x_2,x_3x_4\}$.
\end{itemize}
For example, the graph $X_3$ in shown in Figure~\ref{fig:X3}.

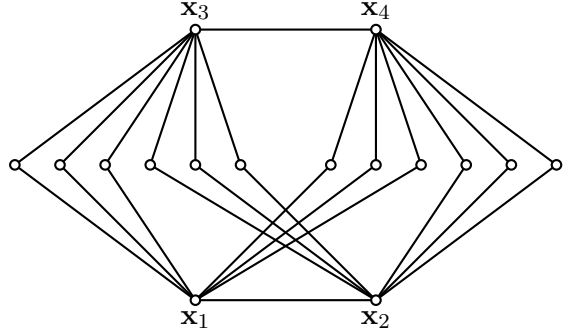
\begin{figure}[ht!]
\begin{center}
\begin{tikzpicture}[scale=0.6,style=thick]
\def\vr{3pt} % \vr = vertex radius;  Set \vr = 2/scale for uniform sizing of vertices
%%%
%%% Vertices:
\path (-2,0) coordinate (0); \path (2,0) coordinate (1);
\path (-2,6) coordinate (2);  \path (2,6) coordinate (3);
\path (-6,3) coordinate (4);  \path (-5,3) coordinate (5);
\path (-4,3) coordinate (6);  \path (-3,3) coordinate (7);
\path (-2,3) coordinate (8);  \path (-1,3) coordinate (9);
\path (1,3) coordinate (10);  \path (2,3) coordinate (11);
\path (3,3) coordinate (12);  \path (4,3) coordinate (13);
\path (5,3) coordinate (14);  \path (6,3) coordinate (15);
%%% Vertices: Horizontal
%%% Draw edges
\draw (0) -- (1); \draw (2) -- (3);
\foreach \i in {4,...,6}
{  \draw (0) -- (\i); \draw (2) -- (\i); }
\foreach \i in {10,...,12}
{  \draw (0) -- (\i); \draw (3) -- (\i); }
\foreach \i in {7,...,9}
{  \draw (1) -- (\i); \draw (2) -- (\i); }
\foreach \i in {13,...,15}
{  \draw (1) -- (\i); \draw (3) -- (\i); }
%%% Draw vertices
\draw (0)  [fill=white] circle (\vr); \draw (1)  [fill=white] circle (\vr);
\draw (2)  [fill=white] circle (\vr); \draw (3)  [fill=white] circle (\vr);
\foreach \i in {4,...,15}
{  \draw (\i)  [fill=white] circle (\vr); }
\draw[anchor = south] (2) node {$\mathbf x_3$}; \draw[anchor = south] (3) node {$\mathbf x_4$};
\draw[anchor = north] (0) node {$\mathbf x_1$}; \draw[anchor = north] (1) node {$\mathbf x_2$};
%\draw[anchor = east] (13) node {$\mathbf a$}; \draw[anchor = east] (14) node {$\mathbf b$};
%\draw[anchor = east] (15) node {$\mathbf c$};
\end{tikzpicture}
\end{center}
\caption{The graph $X_3$}
\label{fig:X3}
\end{figure}

We claim that $i(X_m)=m+2$. The set $D \cup \{x_1,x_3\}$ is an independent dominating set of $X_m$, and hence $i(X_m) \le m+2$.  Now let $J$ be any independent dominating set of $X_m$.  If $J$ has a nonempty intersection with one of $A,B,C$ or $D$, then $J$ contains all the vertices of that set.  On the other hand, no independent subset of $\{x_1,x_2,x_3,x_4\}$ dominates all of $A\cup B\cup C \cup D$.  This establishes the claim.

Let $H_r$ be the complete multipartite graph of order $2r$ in which each of the
$r$ partite sets has cardinality $2$.  More specifically, let the partite sets be $\{u_i,v_i\}$ for $i\in [r]$.  It is straightforward to check that the set $I$ defined by
$I=(\{x_1,x_2\} \times \{u_1,v_1\}) \cup (\{x_3,x_4\} \times \{u_2,v_2\})$ is an independent dominating set of $X_m \times H_r$.  Therefore,
\[i(X_m \times H_r) \le 8 < (m+2)2=i(X_m)i(H_r)\,.\]
This shows that Conjecture~\ref{conj:lowerbound} is false.  Moreover, it shows that the difference $i(G)i(H)-i(G \times H)$ can be arbitrarily large.  In fact,
for $m\ge 7$ we see that $i(X_m \times H_r) < i(X_m)$.

Since the above shows there exist pairs of graphs $G$ and $H$ such that $i(G \times H)$ is not only strictly smaller than $i(G)i(H)$
but can be smaller than $\max\{i(G),i(H)\}$, we are led to the following obvious question.

\begin{ques} \label{ques:1}
Is $i(G\times H) \ge \min\{i(G),i(H)\}$ for every pair of graphs $G$ and $H$?
\end{ques}
\medskip

We now prove Theorem~\ref{thm:extremecounterexample}, which answers the above question in the negative and in fact shows
that the difference $\min\{i(G),i(H)\}- i(G\times H)$ can be arbitrary large.

\medskip
\noindent \textbf{Theorem~\ref{thm:extremecounterexample}} \emph{
For any positive integer $n$ such that $n>10$, there exists a pair of graphs $G$ and $H$ such that
$\min\{i(G),i(H)\}=n+2$ and $i(G \times H) \leq 12$.
}
\begin{proof}
For each positive integer $n$ such
that $n>10$, we now define a pair of graphs $G_n$ and $H_n$.

Let ${\cal A}$ be the collection of subsets of $[6]$ defined by
\[{\cal A}= \{\{3, 4, 5, 6\}, \{2, 5, 6\}, \{1, 2, 3, 4\}, \{1, 3, 4, 6\}, \{1, 2, 5\}\}\,,\]
and let $A_s=\{u_s,v_s\}$, for each $s \in [6]$.  For each $J \in {\cal A}$, we let $A_J$ be an independent set of $n$ vertices.
The graph $G_n$ has vertex set
\[V(G_n) = \left(\bigcup_{s=1}^6 A_s\right) \cup \left(\bigcup_{J\in{\cal A}} A_J\right )\,.\]
The only edges of $G_n$ are given by the following three conditions.
\begin{itemize}
\item For each $s \in [6]$, the vertex $u_s$ is adjacent to $v_s$.
\item For each $J \in {\cal A}$ and for every $s \in J$, each of the $n$ vertices of $A_J$ is adjacent to both vertices of $A_s$.
\item Each of the sets $A_1 \cup A_5$, $A_1 \cup A_6$, $A_2 \cup A_3$, $A_2 \cup A_4$,  $A_2 \cup A_6$,  $A_3 \cup A_5$,  and $A_4 \cup A_5$ induces a clique in $G_n$.
\end{itemize}

We claim that $i(G_n) = n+2$. To see this, observe first that $\{u_1, u_2\} \cup A_{\{3, 4, 5, 6\}}$  is an independent dominating set of $G_n$.
To see that $i(G_n) \ge n+2$, let $X = \cup_{i=1}^6 A_i$.  It is easy to see that the only maximal independent sets in $G_n[X]$ are the following:
\begin{enumerate}
\item[(a)] $\{x, y\}$ where $x \in A_1$ and $y \in A_2$,
\item[(b)] $\{x, y, z\}$ where $x \in A_1$, $y\in A_3$, and $z \in A_4$
\item[(c)] $\{x, y\}$ where $x \in A_2$ and $y \in A_5$
\item[(d)] $\{x, y, z\}$ where $x \in A_3$, $y\in A_4$, and $z \in A_6$
\item[(e)] $\{x, y\}$ where $x \in A_5$ and $y \in A_6$
\end{enumerate}

Moreover, for each maximal independent set $I$ of $G_n[X]$ listed above, there exists a $J \in {\cal A}$ such that
no vertex of $A_J$ is adjacent to a vertex of $I$. Thus, $i(G_n) \ge n+2$.
\vskip5mm
The graph $H_n$ is defined in a similar way.
Let ${\cal B}$ be the collection of subsets of $[6]$ defined by
\[ {\cal B}= \{\{2, 3, 4, 6\}, \{2, 3, 4, 5\}, \{1, 3, 5, 6\}, \{1, 2, 4, 6\}, \{1, 3, 4, 5\}, \{1, 2, 3, 6\}, \{1, 4, 5, 6\}\}\,.\]
For each $K \in {\cal B}$ we let $B_K$ be an independent set of $n$ vertices.
The vertex set of $H_n$ is given by
\[V(H_n) = \{y_1, y_2, y_3, y_4, y_5, y_6\} \cup \left(\bigcup_{K\in{\cal B}} B_K\right )\,,\]
and the edge set of $H_n$ is given by the following two conditions.

\begin{itemize}
\item $\{y_1y_2, y_1y_3, y_1y_4, y_2y_5, y_3y_6, y_3y_4, y_4y_6, y_5y_6\} \subset E(H_n)$
\item For every $K \in {\cal B}$, the vertex $y_k$ is adjacent to each vertex of $B_K$ if and only if $k \in K$.
\end{itemize}

We claim that $i(H_n) = n+2$. One can easily verify that the only maximal independent sets in the induced subgraph
$H_n[\{y_1, y_2, y_3, y_4, y_5, y_6\}]$ are the following: $\{y_1, y_5\}$, $\{y_1, y_6\}$, $\{y_2, y_4\}$, $\{y_3, y_5\}$,  $\{y_2, y_6\}$,
$\{y_2, y_3\}$, and $\{y_4, y_5\}$.

Moreover, for each maximal independent set $I$ of $H_n[\{y_1, y_2, y_3, y_4, y_5, y_6\}]$ listed above, there exists a set $B_K$ such that no vertex of $B_K$
is adjacent either vertex of $I$. Thus, $i(H_n) \ge n+2$. On the other hand, $\{y_1, y_5\} \cup B_{\{2, 3, 4, 6\}}$ is an independent dominating set of $H_n$.

Therefore, we have shown that $i(G_n)=i(H_n)=n+2$.  We claim that the set $D$ defined by $D = \cup_{s=1}^6 \{(u_s, y_s), (v_s,y_s)\}$
is an independent dominating set of $G_n \times H_n$.
It is clear that $\{(u_s, y_s), (v_s,y_s)\}$ is independent for each $s \in [6]$.  Now suppose $(a, y_j)$ and $(b, y_k)$ are adjacent where $a \in A_j$ and
$b \in A_k$  for $1 \le j\le k \le 6$.  It follows that $ab \in E(G_n)$ and $y_jy_k \in E(H_n)$. However, by construction, each vertex of $A_j$ is adjacent to each vertex of $A_k$ only if
$\{y_j, y_k\}$ is an independent set in $H_n$, which is a contradiction.  Hence, $D$ is independent in $G_n \times H_n$.

 Now we verify that $D$ dominates $G_n\times H_n$. First, we show that all vertices of $X\times \{y_1, y_2, y_3, y_4, y_5, y_6\}$ are dominated.
\begin{itemize}
\item $A_1 \times \{y_1\}$ dominates $A_1 \times \{y_1, y_2, y_3, y_4\}$, $A_5 \times \{y_2, y_3, y_4\}$ and $A_6 \times \{y_2, y_3, y_4\}$.
\item $A_2 \times \{y_2\}$ dominates $A_2 \times \{y_1, y_2, y_5\}$, $A_3 \times \{y_1, y_5\}$, $A_4\times \{y_1, y_5\}$, and $A_6\times \{y_1, y_5\}$.
\item $A_3 \times \{y_3\}$ dominates $A_3 \times \{y_1, y_3, y_4, y_6\}$, $A_2\times \{y_1, y_4, y_6\}$, and $A_5\times\{y_1, y_4, y_6\}$.
\item $A_4 \times \{y_4\}$ dominates $A_4\times \{y_1, y_3, y_4, y_6\}$ and $A_2 \times \{y_3\}$.
\item $A_5 \times \{y_5\}$  dominates $A_1 \times \{y_6\}$, $A_3 \times \{y_2\}$, and $A_4\times \{y_2\}$.
\item $A_6 \times \{y_6\}$  dominates $A_1\times \{y_5\}$.
\end{itemize}

Next, let $J \in {\cal A}$ and let $g \in A_J$.  It is easy to see that $\{y_j: j \in J\}$ is a total dominating set of $H_n$.  This implies that
$\cup_{j \in J} \{(u_j,y_j),(v_j,y_j)\}$ dominates $\LSs {g}{H_n}$.
Finally, let $K \in {\cal B}$ and let $h \in B_K$.  Again it is straightforward to verify that $\cup_{k \in K}A_k$ totally dominates $G_n$.
It follows that $\cup_{k \in K} \{(u_k,y_k),(v_k,y_k)\}$ dominates $G_n^h$.

Therefore, $D$ is an independent dominating set of $G_n \times H_n$ and
\[i(G_n \times H_n) \le |D|=12 < n+2=\min\{i(G_n),i(H_n)\}\,.\]
\end{proof}

\section{Conclusion}
Nowakowski and Rall posited the following list of conjectures involving a direct or Cartesian product in \cite{nr-1996}.
\begin{conj} {\rm \cite[Section 2.4]{nr-1996}} For all graphs $G$ and $H$
\begin{enumerate}
\item[1.] $ir(G\Box H) \ge ir(G)ir(H)$
\item[2.] $i(G\times H) \ge i(G)i(H)$
\item[3.] $\gamma(G\Box H) \ge \gamma(G)\gamma(H)$ (Vizing's conjecture)
\item[4.] $\Gamma(G\times H) \ge \Gamma(G)\Gamma(H)$; $\Gamma(G\Box H) \ge \Gamma(G)\Gamma(H)$
\end{enumerate}
\end{conj}

Bre\v{s}ar proved that $\Gamma(G\Box H) \ge \Gamma(G)\Gamma(H)$ in \cite{b-2005} and Bre\v{s}ar, Klav\v{z}ar, and Rall proved that $\Gamma(G\times H) \ge \Gamma(G)\Gamma(H)$ in \cite{bkr-2007}. It is still unknown whether $ir(G\Box H) \ge ir(G)ir(H)$ ($ir$ denotes the lower irredundance number), and Vizing's conjecture remains unsettled. In this paper, we proved that there exist pairs of graphs for which $i(G\times H) < \min\{i(G), i(H)\}$. We also studied the behavior of $i(G\times K_n)$ for a general graph $G$ and were able to provide the exact values for $i(G\times K_n)$ when $G \in \{P_m, C_m\}$.

Consider the following computational problem.
\begin{center}
\fbox{\parbox{0.85\linewidth}{\noindent
{\sc Independent Domination of Direct Products}\\[.8ex]
\begin{tabular*}{.93\textwidth}{rl}
{\em Input:} & A graph $G$, a positive integer $n \geq 3$ and an integer $k$.\\
{\em Question:} & Is $i(G \times K_n) \leq k$?
\end{tabular*}
}}
\end{center}

As presented in Section~\ref{sec:productwithcomplete}, showing that $i(G \times K_n) \leq k$ is equivalent to finding a weak partition $V_0,V_1,\ldots,V_n,V_{[n]}$  of $V(G)$
that satisfies the four conditions necessary to construct an independent dominating set such that the weight is at most $k$.  We pose the following problem.

\begin{prob}
Determine the complexity of {\sc Independent Domination of Direct Products}
\end{prob}

\end{document}